\documentclass[a4paper]{article}
\usepackage{amsthm,amssymb,amsmath,enumerate,graphicx}

\usepackage{authblk}

\newtheorem{definition}{Definition}

\newtheorem{theorem}[definition]{Theorem}

\newtheorem{lemma}[definition]{Lemma}
\newtheorem{conjecture}[definition]{Conjecture}

\newcommand{\comment}[1]{}

\newcommand{\A}{\mathcal A}
\newcommand{\B}{\mathcal B}

\newcommand{\sm}{\setminus}

\newcommand{\mstab}{\mathcal{A}}
\newcommand{\rickety}{rare}

\title{The graph formulation of the union-closed sets conjecture}
\author[1]{Henning Bruhn}
\author[2]{Pierre Charbit}
\author[1]{Oliver Schaudt}
\author[3]{Jan Arne Telle\thanks{Part of this  research done while visiting LIAFA in 2011}}

\affil[1]{IMJ, Universit\'e Pierre et Marie Curie}
\affil[2]{LIAFA, Universit\'e Paris Diderot}
\affil[3]{Department of Informatics, University of Bergen}
\date{}

\begin{document}
\maketitle

\begin{abstract}
In 1979 Frankl conjectured that in a finite non-trivial union-closed collection 
of sets there has to be an element that belongs to at least half the sets.
We show that this is equivalent to the conjecture that in a finite non-trivial graph 
there are two adjacent vertices each belonging
to at most half of the maximal stable sets.
In this graph formulation other special cases become natural. 
The conjecture is trivially true for non-bipartite graphs and we show that
it holds also for the classes of chordal bipartite graphs, subcubic bipartite graphs, bipartite series-parallel graphs 
and bipartitioned circular interval graphs.
\end{abstract}

\section{Introduction} 

A set $\mathcal X$ of sets is \emph{union-closed}
if $X,Y\in\mathcal X$ implies $X\cup Y\in\mathcal X$.
The following conjecture was formulated by Peter Frankl in 1979 \cite{Frankl}. 

\newtheorem*{ucsc}{Union-closed sets conjecture}
\begin{ucsc}
Let $\mathcal X$ be a finite union-closed set of sets
with $\mathcal X \neq \{ \emptyset\}$. Then there is a $x\in\bigcup_{X\in\mathcal X}X$
that lies in at least half of the members of $\mathcal X$.
\end{ucsc}

In spite of a great number of papers, see e.g.\ the good bibliography of Markovi\'{c} \cite{M2007} for papers up to 2007,
this conjecture is still wide open. Several special cases are known to hold,
for example when $|\bigcup_{X\in\mathcal X}X|$ is upper bounded, with current best being 11 by Bo\v{s}njak and Markovi\'{c} \cite{Bosnjak},
or when $|\mathcal X|$ is upper bounded, with current best being 46. 
This follows from a lemma by Lo~Faro \cite{LF94b}, 
and independently by Roberts and Simpson~\cite{Simp}. 
The conjecture also holds when certain sets are present in $\mathcal X$, such as a set of size 2 as shown by Sarvate and Renaud \cite{SR89}.
Possibly as a reflection of its general difficulty, Gowers \cite{Gowers} suggested that work on this conjecture 
could fruitfully be done as a collaborative Polymath project.
See~\cite{franklsurvey} for a survey of the literature on the union-closed sets conjecture.

Various equivalent formulations have been discovered. 
We mention in particular Poonen \cite{Poo92} who translates 
the conjecture into the language of lattice theory. 
Several subsequent results together with their proofs 
belong to lattice theory, for example Reinhold~\cite{Reinhold} who proves 
this conjecture for lower semimodular lattices.
A version of the conjecture is also known for hypergraphs; see El-Zahar~\cite{EZ97}. 

In this paper we give a formulation of the conjecture in the language of graph theory.
A set of vertices in a graph is \emph{stable} if no
two vertices of the set are adjacent.
A stable set is \emph{maximal} if it is  
maximal under inclusion, that is, every vertex outside has a neighbour
in the stable set.

\begin{conjecture}\label{graphconj}
Let $G$ be a finite graph with at least one edge. 
Then there will be two adjacent vertices each belonging to at most half of the maximal stable sets. 
\end{conjecture}

Note that Conjecture \ref{graphconj} is true for non-bipartite graphs.
Indeed, if vertices $u$ and $v$ are adjacent there is no stable set containing them both and so one of them
must belong to at most half of the maximal stable sets. An odd cycle will therefore imply
the existence of two adjacent vertices each belonging to at most half of the maximal stable sets. 
The conjecture is for this reason open only for bipartite graphs.
Moreover, in a connected bipartite graph, for any two vertices $u$ and $v$ in different bipartition classes we have
a path from $u$ to $v$ containing an odd number of edges, so that if $u$ and $v$ each belongs to at most half the maximal stable sets
there will be two adjacent vertices each belonging to at most half the maximal stable sets.
Conjecture \ref{graphconj} is therefore equivalent to the following.

\begin{conjecture}\label{mstabconj}
Let $G$ be a finite bipartite graph with at least one edge. 
Then each of the two bipartition classes contains a vertex 
belonging to at most half of the maximal stable sets. 
\end{conjecture}

%
%
%
%

In this paper we show that Conjectures \ref{graphconj} and \ref{mstabconj} are equivalent to the union-closed sets conjecture.
The merit of this graph formulation is that other special cases become natural, 
in particular subclasses of bipartite graphs.
We show that the conjecture holds for the classes of chordal bipartite graphs
and bipartitioned circular interval graphs, and for  
subcubic and series-parallel bipartite graphs.
Moreover, the reformulation allows to test Frankl's conjecture
in a probabilistic sense: In~\cite{rndfrankl} it 
is shown that almost every random bipartite graph satisfies Conjecture~\ref{mstabconj}
up to any given $\delta>0$, that is, 
almost every such graph contains in each bipartition class a vertex
for which the number of maximal stable sets containing it 
is at most $\tfrac{1}{2}+\delta$ times the total number 
of maximal stable sets.

\medskip
Stable sets are also called independent sets, with the maximal stable sets being exactly the independent dominating sets.
A stable set of a graph is a clique of the complement graph and the graph formulation of the conjecture can also be stated in terms of maximal cliques, instead of maximal stable sets.
The set of all maximal stable sets of a bipartite graph, or rather maximal complete bipartite cliques (bicliques) of the bipartite complement graph, was studied by Prisner~\cite{Prisner} who gave upper bounds on the size of this set, also when excluding certain subgraphs.
More recently, Duffus, Frankl and R\"odl~\cite{Duffus} and Ilinca and Kahn~\cite{Ilinca} investigate the number of maximal stable sets in certain regular and biregular bipartite graphs.
In work related to the graph parameter boolean-width, Rabinovich, Vatshelle and Telle~\cite{RTV} study balanced bipartitions of a graph that bound the number of maximal stable sets.
However, we have not found in the graph theory literature any previous work focusing on the number of maximal stable sets that vertices belong to.

\section{Equivalence of the conjectures}

For a subset $S$ of vertices of a graph we denote by $N(S)$ the set of vertices adjacent to a vertex in $S$.
All our graphs will be finite, and whenever we consider a union-closed set $\mathcal X$
of sets, it will be a finite set, all of whose member-sets will be finite as
well. As Poonen~\cite{Poo92} observed the latter assumption does not restrict
generality, while the conjecture becomes false if $\mathcal X$
is allowed to have infinitely many sets.

We need two easy lemmas. The proof of the first is trivial.
\begin{lemma}\label{structureMSS}
Let $G$ be a bipartite graph with bipartition $U,W$, 
and let $S$ be a maximal stable set. 
Then $S=(U\cap S)\cup (W\sm N(U\cap S))$.
\end{lemma}

\begin{lemma}\label{MSScap}
Let $G$ be a bipartite graph with bipartition $U,W$, 
and let $S$ and $T$  be maximal stable sets. 
Then $(U\cap S\cap T)\cup (W\sm N(S\cap T))$ is a 
maximal stable set.
\end{lemma}
\begin{proof}
Clearly, $R=(U\cap S\cap T)\cup (W\sm N(S\cap T))$ is stable. 
Trivially, any vertex in $W\sm R$ has a neighbour in $R$. A vertex $u$ in 
$U\sm R$ does not lie in $S$ or not in $T$ (perhaps, it is not contained in either),
let us say that $u\notin T$. As $T$ is maximal, $u$ has a neighbour $w\in W\cap T$. 
This neighbour $w$ cannot be adjacent to any vertex in $U\cap S\cap T$ as $T$ is stable. 
So, $w$ belongs to $R$ as well, which shows that $R$ is a maximal stable set.
\end{proof}

For a fixed graph $G$ let us denote by $\mstab$ the set
of all maximal stable sets, and for any vertex
$v$ let us write $\mstab_v$ for the sets of $\mstab$
that contain $v$ and $\mstab_{\overline v}$ for the sets
of $\mstab$ that do not contain $v$. 
Let us call a vertex $v$ \emph{\rickety} if $|\mstab_v|\leq\frac{1}{2}|\mstab|$.

\begin{theorem}
Conjecture~\ref{mstabconj} is equivalent to the union-closed sets
conjecture.
\end{theorem}
\begin{proof}
Let us consider first a union-closed set $\mathcal X\neq\{\emptyset\}$,
which, without restricting generality, we may assume 
to include $\emptyset$ as a member. 
We put $U=\bigcup_{X\in\mathcal X}X$ and
we define a bipartite graph $G$ with vertex set $U\cup\mathcal X$,
where we make $X\in\mathcal X$ adjacent with all $u\in X$.

Now we claim that $\tau:S\mapsto U\sm S$ is a bijection between 
$\mstab$ and $\mathcal X$. First note that 
indeed $\tau(S)\in\mathcal X$ for every maximal stable set: 
Set $A=U\cap S$ and $\mathcal B=\mathcal X\cap S$. 
If $U\subseteq S$ then $U\sm S=\emptyset\in\mathcal X$,
by assumption. So, assume $U\nsubseteq S$, which implies $\mathcal B\neq\emptyset$.
As $S$ is a maximal 
stable set, it follows that $U\sm S=U\sm A=N(\mathcal B)$. 
On the other hand, $N(\mathcal B)$ is just the union of the $X\in S\cap\mathcal X=\mathcal B$, 
which is by the union-closed property equal to a set $X'$ in $\mathcal X$.
To see that $\tau$ is injective note that, by Lemma~\ref{structureMSS}, $S$ is determined by 
$U\cap S$, which in turn determines
$U\sm S$. For surjectivity, consider $X\in\mathcal X$.
We set $A=U\sm N(X)$ and observe that $S=A\cup (\mathcal X\sm N(A))$ is a stable set. 
Moreover, as $X\in \mathcal X\sm N(A)$ every vertex in $U\sm A$ is a neighbour of $X\in S$,
which means that $S$ is maximal.

Now, assuming that Conjecture~\ref{mstabconj} is true, there is an 
\rickety\ $u\in U$, that is, it holds that
$|\mstab_u|\leq\frac{1}{2}|\mstab|$.
Clearly $\mstab$ is the disjoint union of
$\mstab_u$ and of $\mstab_{\overline u}$, so that 
\[
|\tau(\mstab_{\overline u})|=|\mstab_{\overline u}|\geq \frac{1}{2}|\mstab|=\frac{1}{2}|\mathcal X|.
\]
As $u\in \tau(S)\in\mathcal X$ for every $S\in \mstab_{\overline u}$, 
the union-closed sets conjecture follows.

\medskip
For the other direction, 
consider a bipartite graph with bipartition $U,W$
and at least one edge. 
Define $\mathcal X:=\{U\sm S:S\in\mstab\}$, and note that $\mathcal X\neq\{\emptyset\}$
as $G$ has at least two distinct maximal stable sets. 
By Lemma~\ref{structureMSS},
there is a bijection between $\mathcal X$ and $\mstab$. Moreover, it 
is a direct consequence of Lemma~\ref{MSScap} that $\mathcal X$ is union-closed. 
From this, it is straightforward 
that Conjecture~\ref{mstabconj} follows from the union-closed sets conjecture. 
\end{proof}

%
%
%

\section{Application to four graph classes}

For a set $X$ of vertices we define $\mstab_X$ to be the set of maximal stable sets
containing all of $X$. As before, we abbreviate $\mstab_{\{x\}}$ to $\mstab_x$.
\begin{lemma}\label{injection}
Let $x$ be a vertex of a bipartite graph $G$. 
Then there is an injection $\mstab_{N(x)}\to\mstab_x$.
\end{lemma}
\begin{proof}
We define 
\[
i:\mstab_{N(x)}\to\mstab_x,\, S\mapsto S\sm L_1\cup\{x\}\cup (L_2\sm N(S\cap L_3)),
\]
where $L_i$ denotes the set of vertices at distance~$i$ to $x$.
That $i(S)$ is stable and maximal is a direct consequence of the definition.
Moreover, $i(S)=i(T)$ for $S,T\in\mstab_{N(x)}$ implies that $S$ and $T$ 
are identical outside $L_1\cup L_2$.
Moreover, $S$ and $T$ are also identical on $L_1\cup L_2$: First, $L_1=N(x)$
shows that $L_1$ lies in both $S$ and $T$. Second, since every 
vertex in $L_2$ is a neighbour of one in $L_1\subseteq S\cap T$, 
no vertex of $L_2$ can lie in either of $S$ or $T$. 
Thus, $S=T$, and we see that $i$ is an 
injection.
\end{proof}

We denote by $N^2(x)=N(N(x))$ the second neighbourhood of a vertex $x$.
The following lemma generalises the observation that 
if a union-closed set contains a singleton then it satisfies the 
union-closed sets conjecture:
\begin{lemma}\label{onelem}
Let $x,y$ be two adjacent vertices in a bipartite graph $G$
with $N^2(x)\subseteq N(y)$. Then $y$ is \rickety.
\end{lemma}
\begin{proof}
From $N^2(x)\subseteq N(y)$ it follows that every maximal stable set containing $y$
must contain all of $N(x)$. Thus, $\mstab_y=\mstab_{N(x)}$, which means by Lemma~\ref{injection}
that $|\mstab_y|\leq|\mstab_x|$ and as 
$|\mstab_y|+|\mstab_x|\leq |\mstab|$ the lemma is proved.
\end{proof}

We now apply the lemma to the class of \emph{chordal bipartite} graphs.
This is the class of bipartite graphs in which every 
cycle with length at least six has a chord. 

\medskip
This graph class was originally defined in 1978 by Golumbic and Gross \cite{Gol}.
It is also known as the class of bipartite weakly chordal graphs.

A vertex $v$ in a bipartite graph 
is \emph{weakly simplicial} if the neighbourhoods of its neighbours
form a chain under inclusion.  
Hammer, Maffray and Preissmann~\cite{HMP89}, and also Pelsmajer, Tokaz and West~\cite{PTW04} prove the 
following:
\begin{theorem}\label{weaklysimpl}
A bipartite graph with at least one edge is chordal bipartite if and only if every induced subgraph 
has a weakly simplicial vertex. Moreover, 
such a vertex can be found in 
each of the two bipartition classes.
\end{theorem} 

Let us say that a bipartite graph \emph{satisfies Frankl's conjecture}
if each of its bipartition classes contains a \rickety\ vertex.
In order to avoid repeating the trivial condition that the graph has to contain 
at least one edge, we will also consider  edgeless graphs to
satisfy Frankl's conjecture.

\begin{theorem}\label{chbipthm}
Chordal bipartite graphs satisfy  Frankl's conjecture. 
\end{theorem}
\begin{proof}
For a given bipartition class, let $x$ be a weakly simplicial vertex in it. 
Among the neighbours of $x$ denote by $y$ the one whose neighbourhood includes 
the neighbourhoods of all other neighbours of $x$. Then $y$ is \rickety,
by Lemma~\ref{onelem}.
\end{proof}

Going beyond
chordal bipartite graphs, we quickly encounter graphs that cannot be 
handled anymore by Lemma~\ref{onelem}: No vertex in an even cycle 
of length at least six can be proved to be \rickety\ by applying Lemma~\ref{onelem}.
We will, therefore, strengthen the lemma  to at least cover
all even cycles. 

For this, let us extend our notation a bit.
For two vertices $u,v$ let us denote by $\mstab_{uv}$ the set of $S\in\mstab$
containing both of $u$ and $v$, by $\mstab_{u\overline v}$
the set of $S\in\mstab$
containing  $u$ and but not $v$, and by $\mstab_{\overline u\overline v}$
the set of $S\in\mstab$ containing neither of $u$ and $v$.

\begin{lemma}\label{twolem}
Let $G$ be a bipartite graph.
Let $y$ and $z$ be two neighbours of a vertex $x$ so that $N^2(x)\subseteq N(y)\cup N(z)$.
Then one of $y$ and $z$ is \rickety.
\end{lemma}
\begin{proof}
We may assume that $|\mstab_{y\overline z}|\leq |\mstab_{\overline y z}|$.
Now, from $N^2(x)\subseteq N(y)\cup N(z)$ we deduce that $\mstab_{yz}=\mstab_{N(x)}$.
Thus, by Lemma~\ref{injection}, we obtain $|\mstab_{yz}|\leq |\mstab_{x}|$.
Since $\mstab_{x}\subseteq\mstab_{\overline y\overline z}$ it follows 
that 
$|\mstab_{y}|=|\mstab_{y\overline z}|+ |\mstab_{yz}|  \leq
|\mstab_{\overline y z}|+|\mstab_{\overline y\overline z}|=
 |\mstab_{\overline y}|$. 
As $|\mstab|=|\mstab_{y}|+ |\mstab_{\overline y}|$, we see that  $y$ is \rickety.
\end{proof}

Again, the lemma generalises a fact that is well known for the 
set formulation of the union-closed sets conjecture: If one of 
the sets in the union-closed set $\mathcal X$ contains exactly two elements 
then one of the two elements will lie in at least half of the members
of $\mathcal X$; see Sarvate and Renaud~\cite{SR89}.

\medskip
Next we give an application of Lemma~\ref{twolem} to a class of 
graphs derived from \emph{circular interval graphs}. 
The class of circular interval graphs plays a fundamental role 
in the structure theorem of claw-free graphs of Chudnovsky and Seymour~\cite{Sey}.
Circular interval graphs are defined as follows:
Let a finite subset of a circle be the vertex set, and 
for a given  set of subintervals of the circle consider 
two vertices  to be adjacent if there is an interval
containing them both. 
This class is equivalent to what is known as the proper circular arc graphs.

Circular interval graphs are not normally bipartite. 
The only exceptions are
 even cycles and disjoint unions of paths. 
Nevertheless, we may obtain a rich class of bipartite graphs 
from circular interval graphs: For any circular interval graph, 
partition its vertex set and 
delete every edge with both its endvertices in the same class. 
We call any graph arising in 
this manner a \emph{bipartitioned circular interval graph}.

\begin{figure}[ht]
\centering
\includegraphics[scale=0.8]{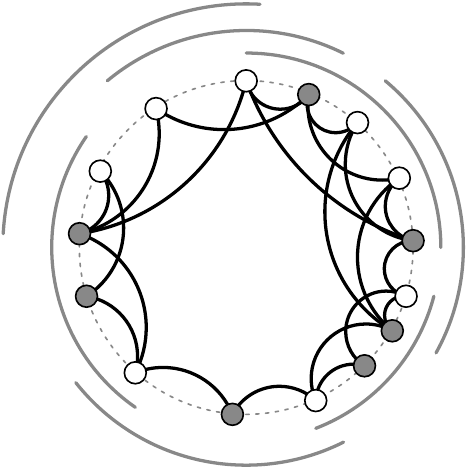}
\caption{A bipartitioned circular interval graph}\label{circint}
\end{figure}

\begin{theorem}
Bipartitioned circular interval graph satisfy Frankl's conjecture.
\end{theorem}
\begin{proof}
Consider a bipartitioned circular interval graph defined by intervals $\mathcal I$,
and let $x$ be a non-isolated vertex of the graph. 

For every neighbour $u$ of $x$ 
we choose an interval $I_u\in\mathcal I$ containing both $x$ and $u$. 
If $\bigcup_{v\in N(x)}I_v$ covers the whole circle, then 
there are already two such intervals $I_y$ and $I_z$ that cover the circle. 
Clearly, every vertex not in the same bipartition class
as $y$ and $z$ is adjacent to at least one of them. 
In particular, $N^2(x)\subseteq N(y)\cup N(z)$. 

So, let us assume that there is a point $p$ on the circle that is not covered
by any $I_v$, $v\in N(x)$. We choose $y$ as the first
neighbour of $x$ from $p$ in clockwise direction, and $z$ as the first 
neighbour of $x$ from $p$ in counterclockwise direction. Then 
$y,v,z$ appear in clockwise order for every $v\in N(x)$
and $v'\in I_y\cup I_z$ 
for every vertex $v'$ so that $y,v',z$ appear in clockwise order.

Let us show that again $N^2(x)\subseteq N(y)\cup N(z)$. For this consider a $u\in N^2(x)$, 
and a neighbour $w$ of $x$ that is adjacent to $u$. Thus, there is a $J\in\mathcal I$
containing both $u$ and $w$. If $y,u,z$ appear in clockwise order, then $u\in I_y\cup I_z$, 
which implies $u\in N(y)\cup N(z)$. If not, then $J$ meets one 
of $y$ or $z$ as $y,w,z$ appear in clockwise order. 
Thus, by virtue of $J$, the vertex $u$ is adjacent to at least one of $y$ and $z$.

In both cases, we apply Lemma~\ref{twolem} in order to see that one of $y$ and $z$ is 
\rickety. As the choice of $x$ was arbitrary, we find \rickety\ vertices in 
both bipartition classes.
\end{proof}

%

Let us now turn to \emph{subcubic} bipartite graphs:
Bipartite graphs in which no vertex has a degree greater than~$3$.

\begin{theorem}\label{thm:MaxDeg3}
Subcubic bipartite graphs satisfy
Frankl's conjecture.
\end{theorem}

Our proof of Theorem~\ref{thm:MaxDeg3} needs some preparation.
Let us call a graph $G$ \emph{reduced} if there is no vertex $v$ whose neighbourhood is equal to the union of neighbourhoods 
of some other vertices.
In particular, reduced graphs are \emph{twin-free}, that is, no two vertices have identical neighbourhoods.
The following lemma tells us that we may restrict our attention to reduced bipartite graphs.

\begin{lemma}\label{reducedlem}
For any bipartite graph $G$ there is a reduced induced subgraph $G'$
so that $G$ satisfies Frankl's conjecture if $G'$ satisfies it.
\end{lemma}
\begin{proof}
Assume there are pairwise distinct vertices $u , v_1 , v_2, \ldots , v_k$ such that $N(u) = \bigcup_{i=1}^k N(v_i)$.
Then $\A_u = \A_{\{v_1 , v_2, \ldots , v_k\}}$. 
Thus, if $A$ is a maximal stable set of $G$, 
then $A-u$ is one of $G-u$, and conversely, any maximal stable 
set $A'$ of $G-u$ is already maximally stable in $G$ if $\{v_1 , v_2, \ldots , v_k\} \not \subseteq A'$; 
otherwise $A'+u$ is a maximal stable set of $G$. 
Hence, a rare vertex of $G-u$ is also rare in $G$.
The assertion is now obtained by iteratively
deleting vertices such as $u$  from $G$.
\end{proof}

Unlike the two classes above, subcubic 
graphs do not have an easily exploitable local structure.
In particular, Lemmas~\ref{onelem} and~\ref{twolem} will have only limited use.
Nevertheless, we can verify Frankl's conjecture by adapting two results 
on the set formulation of the union-closed sets conjecture 
into the graph setting. Both results, one of Vaughan and the other of Knill,
have surprisingly involved proofs.
For a union-closed set $\mathcal X$, we say that 
an element of $\bigcup\mathcal X$ is \emph{abundant} if the element appears
in at least half of the member-sets of $\mathcal X$.

\begin{theorem}[Vaughan~\cite{Vaug04}]\label{Vaughan}
Let $\mathcal X$ be a union-closed set containing three distinct sets of size~$3$
all of which have one element in common.
Then there is an abundant element in the union of the three sets.
\end{theorem}

While Vaughan's theorem gives a local condition, not unlike Lemmas~\ref{onelem}
and~\ref{twolem}, when a particular union-closed set satisfies the conjecture, 
the following result of Knill treats a special class of union-closed sets, 
which he calls \emph{graph-generated families}. 
In this context, we view edges of a graph $H$ as subsets of $V(H)$ of size two.

\begin{theorem}[Knill~\cite{Kni94}]\label{Knill}
Given a graph $H$ with at least one edge, let $\B =\{\bigcup F : F \subseteq E(H) \}$.
Then there is an edge $e \in E(H)$ such that $|\{S \in \B : e \subseteq S\}| \le \tfrac{|\B|}{2}$.
\end{theorem}

Probably unaware of Knill's result, it was restated as a conjecture by El-Zahar~\cite{EZ97}.
Finally, as a response to El-Zahar's paper, it was reproven by Llano, Montellano-Ballesteros, Rivera-Campo and Strausz~\cite{LMRS08}.

We first translate Knill's theorem to the graph setting:

\begin{lemma}\label{cor:Knill}
Let $G$ be a twin-free bipartite graph with bipartition $U \cup W$, 
where every vertex in $U$ is of degree~$2$.
Then there is a \rickety~vertex in $U$.
\end{lemma}

\begin{proof}
Again, let $\A$ be the set of maximal stable sets of $G$.
Observe that $G$ is the subdivision of the graph $H$ on vertex set $W$, 
where any two distinct vertices $x,y$ of $H$ are adjacent 
if and only if they have a common neighbor $u\in U$ in $G$.
As $G$ is twin-free, 
every edge $e=xy$ of $H$ corresponds to a unique vertex $u_e \in U$ with $N(u_e)=\{x,y\}$. 

Let $\B =\{\bigcup F : F \subseteq E(H) \}$, and note that $\B=\{N_G(U'):U'\subseteq U\}$.
We will establish a bijection between $\B$ and $\A$.
For this, 
denote by $\A_{\cap W}$ the intersections of maximal stable sets of $G$ with $W$. 
Then we define the mapping $\B\to\A_{\cap W}$ by $N_G(U')\mapsto W\sm N_G(U')$,
for $U'\subseteq U$. 
As $(W\sm N_G(U'))\cup (U\sm N_G(W\sm N_G(U')))$ is a maximal stable set, the mapping is 
a bijection. 
Recall that Lemma~\ref{structureMSS} asserts
that every maximal stable set is determined by its intersection with one of the bipartition 
classes. Thus, the bijection $\B\to\A_{\cap W}$ extends to a bijection $\B\to\A$.
In particular, $|\A|=|\B|$.

Now, for any $S \in B$ there exists $U' \subseteq U$ so that $N_G(U')=S$.
Any edge $e\in E(H)$ between vertices $x,y\in W$ is contained in $S$ if and only if $x,y \notin W \ N_G(U')$, which means that
the unique maximal stable set $A\in\mathcal A$ with $A\cap W=W\sm N_G(U')$
needs to contain $u_e$, the vertex in $U$ with neighbours $x,y$. 
Therefore, the number of $S\in\B$ with $e\subseteq S$
is equal to the number of maximal stable sets containing $u_e$. 

Applying Theorem~\ref{Knill} we obtain an edge $e=xy \in E(H)$ such that $|\{S \in \B : \{x,y\} \subseteq S\}| \le \tfrac{|\B|}{2}$. This then implies that $u_e$ lies in at most
$\tfrac{|\B|}{2}=\tfrac{|\A|}{2}$ maximal stable sets, which completes the proof.
\end{proof}


\begin{proof}[Proof of Theorem~\ref{thm:MaxDeg3}.]
Let $G$ be a subcubic bipartite graph with bipartition $U \cup W$, 
and let $\A$ be the set of maximal stable sets of $G$.
By Lemma~\ref{reducedlem}, we may assume that $G$ is reduced and, in particular, twin-free.

Let us prove that there is a \rickety~vertex in $U$.
Then, by symmetry, we know that there must be a \rickety~vertex in $W$ too.
If $W$ contains a vertex of degree~$1$ or~$2$, we are done by Lemma~\ref{twolem}.
So, let us assume that every vertex in $W$ has degree~$3$.

First assume that there is a vertex $u \in U$ of degree~$1$.
Let $x\in W$ be its unique neighbor, and let $y,z\in U$ be the other two neighbors of $x$. 
By Lemma~\ref{twolem}, $y$ or $z$ is \rickety~and we are done.

Now assume that there is a vertex $u \in U$ of degree~$3$, say $N(u) = \{x,y,z\}$.
Consider the  set $\B= \{U \setminus S : S \in \A\}$, which is union-closed
by Lemma~\ref{MSScap}.
Then $N(x), N(y), N(z) \in \B$, and $u \in N(x) \cap N(y) \cap N(z)$.
Note that $N(x), N(y), N(z)$ are three distinct sets as $G$ is twin-free.
From Theorem~\ref{Vaughan} we know that there is an abundant element 
of $\B$ in $N(x) \cup N(y) \cup N(z)$, and hence this is a \rickety~vertex in $U$.

The remaining case, when every vertex in $U$ is of degree~$2$ is taken care of 
by Lemma~\ref{cor:Knill}.
\end{proof}

Recall that a graph is called \emph{series-parallel} if it does not contain $K_4$ as a minor.
Equivalently, a graph is series-parallel if and only if it is of treewidth at most two.
Reusing some of the tools presented above,
we can settle Frankl's conjecture for bipartite series-parallel graphs.

\begin{theorem}\label{thm:series-parallel}
Bipartite series-parallel graphs
satisfy Frankl's conjecture.
\end{theorem}

The following lemma  gives us enough 
information on the local structure of a series-parallel graph to prove the  theorem with Lemmas~\ref{onelem} and~\ref{twolem}.

\begin{lemma}[Juvan, Mohar and Thomas~\cite{JMT99}]\label{lem:SP-configs}
Every non-empty series-parallel graph $G$ has one of the following:

\begin{enumerate}[\rm (a)]
	\item  a vertex of degree at most one,\label{cond:deg1}
	
	\item two twins of degree two,\label{cond:twins}
	
	\item two distinct vertices $u,v$ and two not necessarily distinct vertices $w,z \in V(G) \setminus \{u,v\}$ such that $N(v)=\{u,w\}$ and $N(u) \subseteq \{v,w,z\}$, or\label{cond:path}

	\item five distinct vertices $v_1, v_2, u_1, u_2, w$ such that $N(w) = \{u_1, u_2, v_1, v_2\}$ and $N(v_i) = \{w,u_i\}$ for $i=1,2$.\label{cond:triangles}
\end{enumerate}
\end{lemma}

\begin{proof}[Proof of Theorem~\ref{thm:series-parallel}]
Let $G$ be a non-empty bipartite series-parallel graph, say with bipartition classes $(U,W)$,
and we may assume that $G$ does not contain any isolated vertex.
Our argumentation is symmetric, so it suffices to show that there is a rare vertex among
the vertices in $U$.
The class of series-parallel graphs is closed under induced subgraphs, 
and thus by Lemma~\ref{reducedlem} we may assume that $G$ is reduced.

Let $L$ be the set of \emph{leaves} of $G$, that is, the set 
of degree~$1$ vertices.
If there is a leaf in $W$, we obtain with Lemma~\ref{onelem}  a rare vertex in $U$.
So we may assume that $L \subseteq U$.
Let $G' = G - L$ be the graph obtained by deleting all leaves. 
Since $L \subseteq U$, every vertex in $U \cap V(G')$ is of degree at least~$2$.
In particular, $G'$ is not empty.

We claim that in $G'$ there is some vertex $x \in W$ of degree at most $2$.
If the claim is true then 
Lemma~\ref{twolem} yields that some $y \in N_{G'}(x)\subseteq U$ is rare in $G$,
since every  neighbour of $x$ in $G-G'$ is a leaf.

So it remains to prove the claim.
Lemma~\ref{lem:SP-configs} yields that $G'$ contains one of 
the configurations in~\eqref{cond:deg1},~\eqref{cond:twins},~\eqref{cond:path}, or~\eqref{cond:triangles}.
Clearly,~\eqref{cond:triangles} is not possible since $G'$ is bipartite and thus triangle-free.

In case~\eqref{cond:deg1}, there is a leaf in $G'$, which then needs to be contained in $W$
because every vertex in $U \cap V(G')$ has degree at least~$2$.
In case~\eqref{cond:twins}, let $u,v$ be the two twins of degree~$2$.
If $u,v \in U$ then $u$ and $v$ are twins in $G$ as well, which is impossible
as $G$ is reduced.
Consequently, $u,v \in W$ and the claim is again verified.
In the last case~\eqref{cond:path}, there are two distinct vertices $u,v$ and two not necessarily distinct vertices $w,z \in V(G) \setminus \{u,v\}$ such that $N(v)=\{u,w\}$ and $N(u) \subseteq \{v,w,z\}$.
But $G'$ is bipartite and so $uw \notin E(G')$.
In particular, both $u$ and $v$ are of degree at most two.
Since $u$ and $v$ are adjacent, one of them is contained in $W$.
This completes the proof.
\end{proof}

\section{Discussion}

Lemmas~\ref{onelem} and~\ref{twolem} generalise the cases when there 
is a vertex $x$ of degree~$1$ or~$2$. Then, one of the neighbours of $x$ 
is \rickety.  
In contrast, the subcubic case required a bit of work.
This is because  
none of the neighbours of a vertex of degree at least~$3$ 
have to be \rickety. 
An example is given in Figure~\ref{twographs} on the left,
where no neighbour of the vertex~$v$
is \rickety. 
Note that both graphs in Figure~\ref{twographs} are subcubic.

Again, this is not new, in the sense that it corresponds 
directly to an observation of Sarvate and Renaud~\cite{SR90} in the set formulation:
A set of size three need not contain any element appearing
in at least half of the member sets of the union-closed set.  

\begin{figure}[ht]
\centering
\includegraphics[scale=0.8]{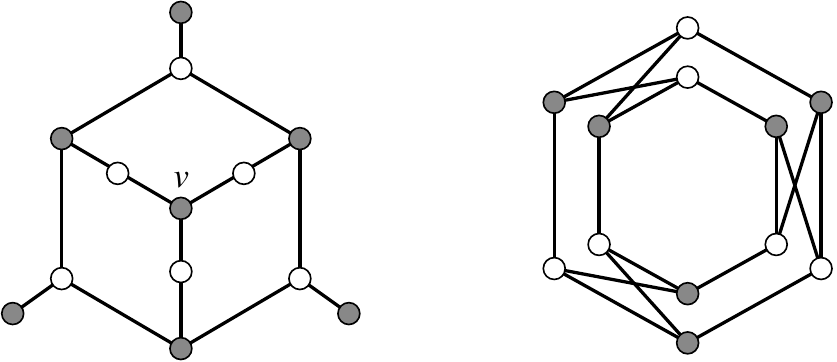}
\caption{Left: No neighbour of $v$ is \rickety. Right: Lemmas~\ref{onelem} or~\ref{twolem}
not applicable}\label{twographs} 
\end{figure}

As chordal bipartite graphs are exactly the $(C_6,C_8,C_{10},\ldots)$-free graphs
one may be tempted to generalise Theorem~\ref{chbipthm} by allowing one more
even cycle, the $6$-cycle, as induced subgraph. 
While Lemma~\ref{onelem} is no longer strong enough even for the $C_6$, 
Lemma~\ref{twolem} easily takes care of any graph with a degree~$2$ 
vertex in each bipartition class. In general, however, Lemma~\ref{twolem}
turns out to be too weak as well to prove the conjecture for 
$(C_8,C_{10},C_{12},\ldots)$-free graphs: 
The graph on the right in Figure~\ref{twographs}
is of that form but has no vertices covered by Lemma~\ref{twolem}.

%

\medskip 
We contend that the results in the previous section substantiate the usefulness
of the graph formulation of the union-closed sets conjecture. 
Moreover, we believe that a good number of other graph classes should be within 
reach. Does Frankl's conjecture hold for planar graphs, regular graphs
or for graphs of treewidth $3$?

\bibliographystyle{amsplain}
\bibliography{franklbib}

\providecommand{\bysame}{\leavevmode\hbox to3em{\hrulefill}\thinspace}
\providecommand{\MR}{\relax\ifhmode\unskip\space\fi MR }
\providecommand{\MRhref}[2]{%
  \href{http://www.ams.org/mathscinet-getitem?mr=#1}{#2}
}
\providecommand{\href}[2]{#2}
\begin{thebibliography}{10}

\bibitem{Bosnjak}
I.~Bo\v{s}njak and P.~Markovi\'{c}, \emph{The 11-element case of {F}rankl's
  conjecture}, Electr. J. Comb. \textbf{15} (2008), R88.

\bibitem{franklsurvey}
H.~Bruhn and O.~Schaudt, \emph{The journey of the union-closed sets
  conjecture}, in preparation.

\bibitem{rndfrankl}
\bysame, \emph{The union-closed sets conjecture almost holds for almost all
  random bipartite graphs}, preprint 2013.

\bibitem{Sey}
M.~Chudnovsky and P.D. Seymour, \emph{Claw-free graphs. {III}. {C}ircular
  interval graphs}, J.~Combin.\ Theory (Series B) \textbf{98} (2008), no.~4,
  812--834.

\bibitem{Duffus}
D.~Duffus, P.~Frankl, and V.~R{\"o}dl, \emph{Maximal independent sets in
  bipartite graphs obtained from boolean lattices}, Eur. J. Comb. \textbf{32}
  (2011), no.~1, 1--9.

\bibitem{EZ97}
M.~El-Zahar, \emph{A graph-theoretic version of the union-closed sets
  conjecture}, J.~Graph Theory \textbf{26} (1997), 155--163.

\bibitem{LF94b}
G.~Lo Faro, \emph{Union-closed sets conjecture: Improved bounds},
  J.~Combin.~Math.~Combin.~Comput. \textbf{16} (1994), 97--102.

\bibitem{Frankl}
P.~Frankl, \emph{Handbook of combinatorics (vol. 2)}, MIT Press, Cambridge, MA,
  USA, 1995, pp.~1293--1329.

\bibitem{Gol}
M.C. Golumbic and C.F. Goss, \emph{Perfect elimination and chordal bipartite
  graphs}, J.~Graph Theory \textbf{2} (1978), 155--163.

\bibitem{Gowers}
T.~Gowers, \emph{Gowers's weblog: Possible future {PolyMath} projects},
  http://gowers.wordpress.com/2009/09/16/possible-future-polymath-projects,
  2009.

\bibitem{HMP89}
P.L. Hammer, F.~Maffray, and M.~Preissmann, \emph{A characterization of chordal
  bipartite graphs}, Rutcor research report, Rutgers University, New Brunswick,
  NJ, 1989.

\bibitem{Ilinca}
L.~Ilinca and J.~Kahn, \emph{Counting maximal antichains and independent sets},
  CoRR \textbf{abs/1202.4427} (2012).

\bibitem{JMT99}
M.~Juvan, B.~Mohar, and R.~Thomas, \emph{List edge-colorings of series-parallel
  graphs}, Electron.\ J.\ Combin. \textbf{6} (1999), 1077--8926.

\bibitem{Kni94}
E.~Knill, \emph{Graph generated union-closed families of sets},
  arXiv:math/9409215v1 [math.CO], 1994.

\bibitem{LMRS08}
B.~Llano, J.J. Montellano-Ballesteros, E.~Rivera-Campo, and R.~Strausz,
  \emph{On conjectures of frankl and el-zahar}, J.~Graph Theory \textbf{57}
  (2008), 344--352.

\bibitem{M2007}
P.~Markovi\'{c}, \emph{An attempt at {F}rankl's conjecture.}, Publications de
  l'Institut Math{\'e}matique. Nouvelle S{\'e}rie \textbf{81(95)} (2007),
  29--43.

\bibitem{PTW04}
M.J. Pelsmajer, J.~Tokaz, and D.B. West, \emph{New proofs for strongly chordal
  graphs and chordal bipartite graphs}, preprint 2004.

\bibitem{Poo92}
B.~Poonen, \emph{Union-closed families}, J.~Combin.\ Theory (Series A)
  \textbf{59} (1992), 253--268.

\bibitem{Prisner}
E.~Prisner, \emph{Bicliques in graphs {I}: {B}ounds on their number},
  Combinatorica \textbf{20} (2000), no.~1, 109--117.

\bibitem{RTV}
Y.~Rabinovich, J.A. Telle, and M.~Vatshelle, \emph{Upper bounds on the boolean
  width of graphs with an application to exact algorithms}, submitted, 2012.

\bibitem{Reinhold}
J.~Reinhold, \emph{Frankl's conjecture is true for lower semimodular lattices},
  Graphs and Combinatorics \textbf{16} (2000), no.~1, 115--116.

\bibitem{Simp}
I.~Roberts and J.~Simpson, \emph{A note on the union-closed sets conjecture},
  Austral.\ J.\ Combin. \textbf{47} (2010), 265--269.

\bibitem{SR90}
D.G. Sarvate and J.-C. Renaud, \emph{Improved bounds for the union-closed sets
  conjecture}, Ars Combin. \textbf{29} (1989), 181--185.

\bibitem{SR89}
\bysame, \emph{On the union-closed sets conjecture}, Ars Combin. \textbf{27}
  (1989), 149--154.

\bibitem{Vaug04}
T.P. Vaughan, \emph{Three-sets in a union-closed family},
  J.~Combin.~Math.~Combin.~Comput. \textbf{49} (2004), 73--84.

\end{thebibliography}

\end{document}